\documentclass[11pt]{article}
\usepackage[margin=2cm]{geometry}
\usepackage{amsthm,bbm}
\usepackage{amsmath}
\usepackage{amssymb}
\usepackage{bm}

\newtheorem{definition}{Definition}

\newtheorem{proposition}{Proposition}
\newtheorem{remark}{Remark}
\newtheorem{lemma}{Lemma}

\begin{document}

\title{\textbf{Characterizing the Functional Density Power Divergence Class}}
\author{Ray, S.$^{1}$, Pal, S.$^{2}$, Kar, S.~K.$^{3}$ and Basu, A.$^{4}$\\$^{1}${\small Stanford University},$^{2}${\small Iowa State University}\\$^{3}${\small University of North Carolina, Chapel Hill}, $^{4}${\small Indian Statistical Institute}}
\date{}
\maketitle

\begin{abstract}
Divergence measures have a long association with statistical inference, machine learning and information theory. The density power divergence 
and related measures have produced many useful 
(and popular) statistical procedures, which provide a good balance between model efficiency on one hand and outlier stability or robustness on the other. The logarithmic density power divergence, 
a particular logarithmic transform of the density power divergence, has also been very successful in producing efficient and stable inference procedures; in addition it has also led to significant demonstrated applications in information theory. The success of the minimum divergence procedures based on the density power divergence and the logarithmic density power divergence (which also go by the names $\beta$-divergence and $\gamma$-divergence, respectively) make it imperative and meaningful to look for other, similar divergences which may be obtained as transforms of the density power divergence in the same spirit. With this motivation we search for such transforms of the density power divergence, referred to herein as the functional density power divergence class. The present article characterizes this functional density power divergence class, and thus identifies the available divergence measures within this construct that may be explored further for possible applications in statistical inference, machine learning and information theory. 
\end{abstract}

\section{Introduction}

Divergence measures have natural and appealing applications in many scientific disciplines including statistics, machine learning and information theory. The method based on likelihood, the canonical approach to inference in statistical data analysis, is itself a minimum divergence method; the
maximum likelihood estimator minimizes the likelihood disparity~\cite{L01}, a version of the Kullback-Leibler divergence. Among the 
different formats of minimum divergence inference, the approach based on the minimization of density-based divergences is of particular 
importance, as in this case the resulting procedures combine a high degree of model efficiency with strong robustness properties. 

The central elements in the present research are the collections of the density-based minimum divergence procedures based on the density power divergence (DPD)
of~\cite{B01}. 
The popularity
and utility of this procedure make it important to study other similar divergences in search of competitive or better statistical (and other)
properties. Indeed, one such divergence that is known to us and one that has also left its unmistakable mark on the area of robust statistical inference
is the logarithmic density power divergence (LDPD); see, e.g.,~\cite{J02},~\cite{F01},~\cite{B02},~\cite{F02}. The applicability of this class of divergences in mathematical information theory has been explored in~\cite{K02},~\cite{K03},~\cite{GB01},~\cite{GB02}.

Both the ordinary DPD and the LDPD belong to the functional density power divergence class that we will define in the next section. 
These two families of divergences have also been referred to as the BHHJ and the JHHB classes, or the type 1 and type 0 classes, or 
the $\beta$-divergence and the $\gamma$-divergence classes; more details about their applications may be found in ~\cite{J02},~\cite{F01},~\cite{C01},~\cite{B02} 
among others. However, while the DPD belongs to the class of Bregman divergences, the LDPD does not. The DPD is also a single-integral, non-kernel divergence \cite{J01}; the LDPD is not a {single-integral} divergence, although it is a non-kernel one.  
The non-kernel divergences 
have also been called \textit{decomposable} in the literature~\cite{B03}. The divergences within the DPD family have been shown to possess 
strong robustness properties in statistical applications. 
The LDPD family is also useful in this respect.

Our basic aim
in this work is to characterize the class of functional density power divergences. Essentially, each functional density power divergence corresponds
to a function with the non-negative real line as its domain. The DPD corresponds to the identity function, while the LDPD corresponds to the log function. 
Within the class of functional density power divergences, we will characterize the class of functions which generate legitimate divergences. In turn, this will provide a characterization of the functional density power divergence class.

\section{The DPD and the LDPD}{\label{dpddef}}
Suppose $(\mathbb{R},\mathcal{B}_{\mathbb{R}}, \mu)$ is a measure space on the real line. Introduce the notation
\begin{equation}
	\label{defn}
	\mathcal{L}_{\mu} := \left\{ f : \mathbb{R} \longrightarrow \mathbb{R} : \; f \geq 0, \;f \;\text{measurable}, \; \int_{\mathbb{R}} f \;d\mu = 1 \right\}.
\end{equation}
A divergence defined on $\mathcal{L} \subseteq \mathcal{L}_{\mu}$ is a non-negative function $D : \mathcal{L} \times \mathcal{L} \longrightarrow [0,\infty] $ with the property 
\begin{equation}
	\label{prop}
	D(g,f) =0 \iff g=f \; a.e.[\mu].
\end{equation}
For the sake of brevity we will drop the dominating measure $\mu$ from the notation; it will be understood that all the integrations and almost sure statements are with respect to this dominating measure. We also suppress the dummy variable of integration from the expression of the integrals in the rest of the paper.

One of the most popular examples of such families of divergences is the {\it density power divergence} (DPD) of~\cite{B01} defined as 
\begin{equation}
	\label{DPD}
	\text{DPD}_{\alpha}(g,f) :=  \int_{\mathbb{R}} f^{\alpha+1} - \Big( 1+\dfrac{1}{\alpha} \Big) \int_{\mathbb{R}} f^{\alpha}g + \dfrac{1}{\alpha}  \int_{\mathbb{R}} g^{\alpha+1},
\end{equation}  
for all $g,f \in \mathcal{L}_{\mu, \alpha}$, where $\alpha$ is a non-negative real number and 
$$  \mathcal{L}_{\mu, \alpha} := \left\{ f \in \mathcal{L}_{\mu} \; : \; \int_{\mathbb{R}} f^{1+\alpha} \;d\mu < \infty \right\}.$$
For $\alpha=0$, the definition is to be understood in a limiting sense as $\alpha \downarrow 0$, and the form of the divergence then turns out to be,
\begin{equation}
	\label{DPD2}
	\text{DPD}_{\alpha=0}(g,f) :=  \int_{\mathbb{R}} g \log \dfrac{g}{f}, \;\; {\rm for~all} \; \; g,f \in \mathcal{L}_{\mu} = \mathcal{L}_{\mu,0},
\end{equation}  
which is actually the likelihood disparity, see~\cite{L01}; it is also a version of the Kullback-Leiber divergence. For $\alpha = 1$, the divergence in~(\ref{DPD}) reduces to the squared $L_2$ distance. It is easy and straightforward to check that the definition in~(\ref{DPD}) satisfies the condition in~(\ref{prop}).

Another common example of a related divergence class is the {\it logarithmic density power divergence} (LDPD) family of~\cite{J02} defined as  
\begin{align}
	\label{defnldpd}
	\text{LDPD}_{\alpha} (g,f) &:= \log \Big(\int_{\mathbb{R}} f^{\alpha+1} \Big) - \Big( 1+\dfrac{1}{\alpha} \Big) \log \Big(\int_{\mathbb{R}} f^{\alpha}g \Big) + \dfrac{1}{\alpha} \log \Big( \int_{\mathbb{R}} g^{\alpha+1} \Big),
\end{align} 
for all $g, f \in \mathcal{L}_{\mu, \alpha}, \alpha \geq 0.$ Its structural similarity with the DPD family is immediately apparent. It is obtained by replacing the identity function on each component of the integral by the log function. This family is also known to produce highly robust estimators with good efficiency. ~\cite{F01} and ~\cite{F02}, in fact, argue that the minimum divergence estimators based on the LDPD are more successful in minimizing the bias of the estimator under heavy contamination in comparison to the minimum DPD estimators. However, also see~\cite{K01} for some counter views. The latter work has, in fact, proposed a new class of divergences which provides a smooth bridge between the DPD and the LDPD families.

\section{The Functional Density Power Divergence}

Further exploration of the divergences within the DPD family leads to the observation that this class of divergences may be extended to a more general family of divergences, called the \textit{functional density power divergence family} having the form 
\begin{align}
	\label{defn2}
	\text{FDPD}_{\varphi,\alpha}(g,f) &:= \varphi \Big(\int_{\mathbb{R}} f^{\alpha+1} \Big) - \Big( 1+\dfrac{1}{\alpha} \Big) \varphi \Big(\int_{\mathbb{R}} f^{\alpha}g \Big)  + \dfrac{1}{\alpha} \varphi \Big( \int_{\mathbb{R}} g^{\alpha+1} \Big), 
\end{align} 
for all $g,f \in \mathcal{L}_{\mu,\alpha}$, where $\varphi : [0, \infty) \longrightarrow [-\infty, \infty]$ 
is a pre-assigned function, $\alpha$ is a non-negative real number and $\text{FDPD}_{\varphi, \alpha}$ is a divergence in the sense of (\ref{prop}). Note that the expression given in~(\ref{defn2}) need not necessarily define a divergence for all $\varphi$ as it does not always satisfy the condition stated 
in Section~\ref{dpddef}. Indeed, it may even not be well-defined for all pairs of densities $f,g \in \mathcal{L}_{\mu, \alpha}$ since $\varphi$ may take the value $\infty$. In the following we will identify the class of functions $\varphi$ for which the quantity defined in (\ref{defn2}) is actually a divergence, thus providing a characterization of the FDPD class.  

Within the FDPD class, the case $\alpha=0$ has to be again understood in a limiting sense and this limiting divergence exists under some constraints on the function $\varphi$. For example, if we assume $\varphi$ is continuously differentiable in an interval around $1$, then the divergence for $\alpha=0$ can be defined as
\begin{equation}
	\label{zero case}
	\text{FDPD}_{\varphi,\alpha=0} (g,f) := \varphi^{\prime}(1) \int_{\mathbb{R}} g \log \dfrac{g}{f}, \;\; {\rm for~all} \; \; g,f \in \mathcal{L}_{\mu},
\end{equation}
where $\varphi^{\prime}$ is the derivative of $\varphi$. Obviously we require $\varphi^{\prime}(1)$ to be positive for the above to be a divergence. Note that, the divergence in~(\ref{zero case}) is actually the likelihood disparity with a different scaling constant and therefore the divergences $\text{FDPD}_{\varphi, \alpha=0}$ are effectively equivalent to the likelihood disparity for inferential purposes. For the DPD and the LDPD, in fact, the scaling constant $\varphi^{\prime}(1)$ equals unity. The characterization of the FDPD, therefore, is not an interesting problem for $\alpha = 0$; 
hence, we will not concern ourselves with the $\alpha=0$ case in the following.

\begin{remark}
	Suppose, $\varphi$ is a strictly increasing and convex function on the non-negative real line. Then it is straightforward to check that the expression defined in~(\ref{defn2}) does indeed satisfy the divergence conditions in Section~\ref{dpddef} and therefore defines a legitimate divergence which belongs to the FDPD class. Note that $\varphi^{\prime}(1)$ is necessarily positive in this case. The identity function which relates to the DPD family belongs to this class of $\phi$ functions.
	
\end{remark}

\begin{remark}
	That the class of  functions described in the previous remark does not completely characterize the FDPDs can be seen by choosing  $\varphi(x) = \log x,$ for all $x \geq 0$, with the convention that $\log 0 := -\infty$. In this case $\varphi$ is a concave function but the corresponding FDPD  also satisfies the divergence conditions and gives rise to the logarithmic density power divergence  (LDPD) family already introduced in Section~\ref{dpddef}. In this case $\varphi^{\prime}(1) = 1$.
\end{remark}


We expect that the members of the FDPD family will possess useful robustness properties and have other information theoretic utilities which could make it interesting to examine these divergences and therefore it is natural to ask whether we can characterize all the functions $\varphi$ which give rise to a divergence in~(\ref{defn2}) and thus can obtain a complete description of the FDPD family. As already indicated, the main objective of this article is to discuss this characterization. 

\section{Characterization of the FDPD family}
In this section we will assume that the dominating measure $\mu$ is actually the Lebesgue measure on the real line and therefore the FDPD is a family of divergences on the space of probability density functions.

Our first result states a general sufficient condition on the function $\varphi$ which will guarantee that $\text{FDPD}_{\varphi,\alpha}$ is a valid divergence for all $\alpha > 0.$

\begin{proposition}
	\label{propn1}
	Suppose $\varphi :[0,\infty) \longrightarrow [-\infty,\infty]$ is a function such that the function $\psi : [-\infty, \infty) \longrightarrow [-\infty,\infty]$ defined as $\psi (x) := \varphi(e^x), \;{\rm for~all}~ x \in [-\infty, \infty)$ is convex and strictly increasing on its domain. Moreover assume that $\psi(\mathbb{R}) \subseteq \mathbb{R}$. Then $\text{FDPD}_{\varphi,\alpha}$ is a valid divergence for each fixed $\alpha > 0$, according to the definition in~(\ref{defn2}).
\end{proposition}

\begin{proof}
	We start by observing that for any $ f,g \in \mathcal{L}_{\mu,\alpha}$, the quantities $\int_{\mathbb{R}} f^{1+\alpha} $ and $\int_{\mathbb{R}} g^{1+\alpha} $ are finite and non-zero. Therefore the expression in~(\ref{defn2}) is well-defined since $\varphi(x) = \psi(\log x) \in \mathbb{R}$ for all $x \in (0, \infty)$. Now, in order to show that $\text{FDPD}_{\varphi,\alpha}$ is a valid divergence, we need to establish that $\text{FDPD}_{\varphi,\alpha} (g,f)$ is non-negative for all choices of $g,f \in \mathcal{L}_{\mu, \alpha},$ and it is exactly zero if and only if $g=f, \; a.e.[\mu].$ For $\alpha > 0$, using the convexity of the function $\psi$ we can conclude that 
	\begin{align}
		\alpha \varphi \Big(\int_{\mathbb{R}} f^{\alpha+1} \Big) +  \varphi \Big( \int_{\mathbb{R}} g^{\alpha+1} \Big) 
		&= \alpha \psi \Big(\log \int_{\mathbb{R}} f^{\alpha+1} \Big)+  \psi \Big(\log \int_{\mathbb{R}} g^{\alpha+1} \Big) \nonumber\\
		\label{convexstep} & \geq  (1+\alpha) \psi \Big( \dfrac{\alpha}{1+\alpha} \log \int_{\mathbb{R}} f^{\alpha+1} + \dfrac{1}{1+\alpha}  \log \int_{\mathbb{R}} g^{\alpha+1} \Big).
	\end{align}
	On the other hand, for $\alpha >0$, using \textit{H\"{o}lder's inequality} on the functions $f^{\alpha}$ and $g$ with dual indices $\Big( \dfrac{1+\alpha}{\alpha}, 1+\alpha \Big)$ we obtain,
	\begin{equation}
		\label{holder}
		{\left( \int_{\mathbb{R}} f^{\alpha+1}\right)}^{\frac{\alpha}{1+\alpha}} {\left( \int_{\mathbb{R}} g^{\alpha+1}\right)}^{\frac{1}{1+\alpha}} \geq \int_{\mathbb{R}} f^{\alpha}g,
	\end{equation}
	which is equivalent to
	\begin{equation}
		\label{holder2}
		\dfrac{\alpha}{1+\alpha} \log \int_{\mathbb{R}} f^{\alpha+1} + \dfrac{1}{1+\alpha}  \log \int_{\mathbb{R}} g^{\alpha+1} \geq \log \int_{\mathbb{R}} f^{\alpha}g.
	\end{equation}
	Expression~(\ref{convexstep}) and~(\ref{holder2}) along with the strict monotonicity of $\psi$ imply that,
	\begin{align}
		\alpha \varphi \Big(\int_{\mathbb{R}} f^{\alpha+1} \Big) +  \phi \Big( \int_{\mathbb{R}} g^{\alpha+1} \Big) 
		& \geq  (1+\alpha) \psi \Big( \dfrac{\alpha}{1+\alpha} \log \int_{\mathbb{R}} f^{\alpha+1} + \dfrac{1}{1+\alpha}  \log \int_{\mathbb{R}} g^{\alpha+1} \Big) \nonumber \\
		\label{inc}& \geq  (1+\alpha) \psi \Big( \log \int_{\mathbb{R}} f^{\alpha}g \Big) \\
		&=  (1+\alpha) \varphi \Big( \int_{\mathbb{R}} f^{\alpha}g \Big),
	\end{align}
	which is equivalent to the statement that $\text{FDPD}_{\varphi,\alpha}(g,f) \geq 0$. For the equality $\text{FDPD}_{\varphi,\alpha}(g,f)=0$ to hold, we must have equality in~(\ref{convexstep}) and~(\ref{inc}). By strict monotonicity of $\psi$, the equality in~(\ref{inc}) implies equality in~(\ref{holder}) which will happen only if $f^{1+\alpha}=g^{1+\alpha}, \; a.e.[\mu]$, which is equivalent to $f=g, a.e.[\mu]$. On the other hand, if $f=g,\; a.e.[\mu]$, then clearly $\text{FDPD}_{\varphi,\alpha}(g,f)=0$ by~(\ref{defn2}). This completes our proof.
\end{proof} 
Now we shall show that the condition on $\varphi$ stated in Proposition~\ref{propn1} is indeed a necessary condition for generating a divergence family, for any fixed $\alpha >0$.

\begin{proposition}
	\label{propn2}
	Fix $\alpha \in (0, \infty)$. Suppose $\varphi : [0,\infty) \longrightarrow [-\infty, \infty]$ is a function such that that $\text{FDPD}_{\varphi,\alpha}$ is a valid divergence. Then the function $\psi : [-\infty, \infty) \longrightarrow [-\infty,\infty]$ defined as $\psi (x) := \varphi(e^x), \;{\rm for~all}~ x \in [-\infty, \infty)$ is convex and strictly increasing on its domain with $\psi(\mathbb{R}) \subseteq \mathbb{R}$.  
\end{proposition}

\begin{proof}
	We shall use the idea of computing the divergence between two appropriate probability density functions and extracting the property of the function $\varphi$ from it. Fix any real $\gamma > -1/(1+\alpha)$ and consider the family of probability densities given by 
	\begin{equation}{\label{def}}
		f_{\theta}(x) := (\gamma + 1) \theta^{-\gamma-1} x^{\gamma} \mathbbm{1}_{(0, \theta)}(x), \; \forall \; x \in \mathbb{R}, \; \theta >0,
	\end{equation}
	where $\mathbbm{1}_A$ denotes the indicator function of the set $A$. These are valid probability densities since $\gamma >-1 $. Easy computations show that
	\begin{equation}
		\label{values}
		\int_{\mathbb{R}} f_{\theta}^{1+\alpha} = \dfrac{(\gamma+1)^{1+\alpha}}{1+\gamma(1+\alpha)}\theta^{-\alpha}, \; \; \forall \; \theta >0 ,
	\end{equation}
	and for any $\theta, \tau >0$
	\begin{equation}
		\label{values1}
		\int_{\mathbb{R}} f_{\theta}^{\alpha}f_{\tau} = \begin{cases}
			\dfrac{(\gamma+1)^{1+\alpha}}{1+\gamma(1+\alpha)} \theta^{-\gamma\alpha-\alpha}\tau^{\gamma\alpha}, \;\;\;\;\; \text{if } \theta > \tau, \\
			\dfrac{(\gamma+1)^{1+\alpha}}{1+\gamma(1+\alpha)} \theta^{1+\gamma-\alpha}\tau^{-\gamma-1}, \;\; \text{if } \theta \leq \tau.
		\end{cases}
	\end{equation}
	Therefore, the property that $\text{FDPD}_{\varphi,\alpha}(g,f) \geq 0$ along with equality if and only if $g=f, a.e.[\mu]$ yields that
	\begin{equation}
		\label{exp2}
		\varphi(C\theta^{-\alpha}) - \Big(1+\dfrac{1}{\alpha}\Big) \varphi(C \theta^{-\gamma\alpha-\alpha}\tau^{\gamma \alpha}) + \dfrac{1}{\alpha} \varphi(C\tau^{-\alpha}) > 0, \;\; {\rm if} \; \theta >\tau >0,
	\end{equation}
	and 
	\begin{equation}
		\label{exp3}
		\varphi(C\theta^{-\alpha}) - \Big(1+\dfrac{1}{\alpha}\Big) \varphi(C\theta^{1+\gamma-\alpha}\tau^{-\gamma-1} ) + \dfrac{1}{\alpha} \varphi(C\tau^{-\alpha}) > 0, \; {\rm if} \; \tau >\theta >0,
	\end{equation}
	where $C:=(\gamma+1)^{1+\alpha}/(1+\gamma(1+\alpha))$. The assertion that the expressions in the left hand sides of~(\ref{exp2}) and~(\ref{exp3}) are well-defined is also part of the implication. Now fix any $x,y \in \mathbb{R}$. If $x>y$, plug in $\theta = C^{1/\alpha}\exp(-x/\alpha)$ and $\tau = C^{1/\alpha}\exp(-y/\alpha)$ in Equation~(\ref{exp3}). Notice that $x>y$ will guarantee that $\theta< \tau$. Therefore we get
	\begin{equation}
		\label{exp4}
		\varphi(e^x) + \dfrac{1}{\alpha}\varphi(e^y) > \Big(1+\dfrac{1}{\alpha}\Big) \varphi \Bigg(\exp \Bigg(\left( 1- \dfrac{\gamma+1}{\alpha}\right)x + \dfrac{\gamma+1}{\alpha}y\Bigg)\Bigg), 
	\end{equation}
	for all $x >y \in \mathbb{R}$, which on simplification yields
	\begin{equation}
		\label{exp5}
		\dfrac{\alpha}{1+\alpha}\psi(x) + \dfrac{1}{1+\alpha}\psi(y) >  \psi  \Bigg(\left( 1- \dfrac{\gamma+1}{\alpha}\right)x + \dfrac{\gamma+1}{\alpha}y\Bigg), \; \forall \; x >y \in \mathbb{R}.
	\end{equation} 
	Similar manipulation with (\ref{exp2}) leads us to the following observation.
	\begin{equation}
		\label{exp5s}
		\dfrac{\alpha}{1+\alpha}\psi(x) + \dfrac{1}{1+\alpha}\psi(y) >  \psi  \Big(\left( 1+\gamma\right)x  -\gamma y\Big),  \; \forall \; x <y \in \mathbb{R}.
	\end{equation} 
	
	We shall now proceed with some appropriate choices for $\gamma$. If we take $\gamma=0$ in (\ref{exp5s}), we obtain that $\psi$ is strictly increasing on $\mathbb{R}$. To prove that $\psi$ is indeed strictly increasing on $[-\infty, \infty)$, take $f=\theta^{-1}\mathbbm{1}_{(0,\theta)}$ and $g=\theta^{-1}\mathbbm{1}_{(\theta,2\theta)}$ for some $\theta >0$.  In this case,
	$$ \int_{\mathbb{R}} f^{\alpha}g =0, \; \int_{\mathbb{R}} f^{\alpha+1} =\int_{\mathbb{R}} g^{\alpha+1} = \theta^{-\alpha},$$
	and hence
	$$ 0 < \text{FDPD}_{\varphi,\alpha}(f,g) = \left(1 + \dfrac{1}{\alpha} \right) \left( \varphi(\theta^{-\alpha})-\varphi(0)\right). $$
	Since this holds for all $\theta>0$, we have our required strict monotonicity of $\varphi$ on $[0,\infty)$, which proves the fact that $\psi$ is strictly increasing on $[-\infty,\infty)$. Observe that, strict monotonicity of $\psi$ on $\mathbb{R}$ implies $\psi(\mathbb{R}) \subseteq \mathbb{R}.$ All that remains to show now is the convexity of the function $\psi$.

	Fix any $x>y\in \mathbb{R}$ and take $\gamma=\gamma_n:=-(1+\alpha)^{-1}+1/n$ in~(\ref{exp5}). Since,
	$$ \left( 1- \dfrac{\gamma_n+1}{\alpha}\right)x + \dfrac{\gamma_n+1}{\alpha}y = \dfrac{\alpha}{1+\alpha} x + \dfrac{1}{1+\alpha} y - \dfrac{x-y}{\alpha n} \uparrow \dfrac{\alpha}{1+\alpha} x + \dfrac{1}{1+\alpha} y,\; \text{ as } n \to \infty, $$
	we can conclude that 
	\begin{equation}
		\label{exp6}
		\dfrac{\alpha}{1+\alpha}\psi(x) + \dfrac{1}{1+\alpha}\psi(y) \geq   \psi  \Bigg[\Bigg( \dfrac{\alpha}{1+\alpha} x + \dfrac{1}{1+\alpha} y \Bigg)-\Bigg], \; \forall \; x >y \in \mathbb{R},
	\end{equation} 
	where $\psi(u-) := \lim_{v \uparrow u} \psi(v)$, for all $u \in \mathbb{R}$; which exists since $\psi$ is monotone. Similar manipulation with~(\ref{exp5s}) yields the inequality in~(\ref{exp6}) for $x<y \in \mathbb{R}$. Monotonicity of $\psi$ also guarantees that $\psi(\cdot-)$ is finite on $\mathbb{R}$. Fix $x,y \in \mathbb{R}$ and get sequences $x_n \uparrow x$ and $y_n \uparrow y$. Define $z_n := \alpha(1+\alpha)^{-1} x_n + (1+\alpha)^{-1}y_n -1/n$, for all $n \geq 1$. Clearly, $z_n \uparrow z:=\alpha(1+\alpha)^{-1} x + (1+\alpha)^{-1}y$ and 
	\begin{equation}
		\label{exp7}
		\dfrac{\alpha}{1+\alpha}\psi(x_n) + \dfrac{1}{1+\alpha}\psi(y_n) \geq   \psi  \Bigg[\Bigg( \dfrac{\alpha}{1+\alpha} x_n + \dfrac{1}{1+\alpha} y_n \Bigg)-\Bigg] \geq \psi(z_n), \; \forall \; n \geq 1.
	\end{equation}
	Taking $n \to \infty$ in~(\ref{exp7}), we can conclude that 
	$$ \dfrac{\alpha}{1+\alpha}\psi(x-) + \dfrac{1}{1+\alpha}\psi(y-) \geq \psi(z-),$$
	implying that $\psi(\cdot-)$ is indeed $\alpha/(1+\alpha)$-convex, see Definition~\ref{lambd} in the Appendix. The function $\psi(\cdot-)$, being finite and non-decreasing on $\mathbb{R}$, is  bounded on any finite interval. Applying Lemma~\ref{lem1} and Proposition~\ref{lambdacon:main}, we can conclude that $\psi(\cdot-)$ is convex and continuous on $\mathbb{R}$. Lemma~\ref{lem3} yields that $\psi(\cdot-)=\psi$; hence $\psi$ is convex on $\mathbb{R}$. Monotonicity of $\psi$ on $[-\infty, \infty)$ guarantees that it is indeed convex on $[-\infty, \infty)$. This completes the proof.

\end{proof}

\begin{remark}
	The proof of Proposition~\ref{propn2} does not assume continuity of $\varphi$ (or equivalently of $\psi$) a priori. Instead we have proved that $\psi$ is convex on $\mathbb{R}$, implying that it  also has to be continuous on $\mathbb{R}$.
\end{remark}

\begin{remark}
	Note that if $\varphi$ is convex and strictly increasing then $\psi$ defined as in Proposition~\ref{propn1} and Proposition~\ref{propn2} is strictly convex and strictly increasing. On the other hand, if $\varphi=\log$, then $\psi$ is  the identity function and therefore convex and strictly increasing.
\end{remark}

\begin{remark}
	There exists other directions for proving the necessity part in Proposition~\ref{propn2} assuming some smoothness conditions for the function $\varphi$. One such direction may be provided by the method of  \cite{J01}. Any kind of smoothness assumption being redundant in our proof makes our characterization complete. In fact, it appears that the approach in \cite{J01} might be refined by the approach in the present paper, rather than the other way around. (This was also suggested by one of the reviewers). We hope to explore this in our future work. 
	
	However from a practical point of view and for large sample consistency or influence function calculations, we would probably need some differentiability conditions on $\varphi$.
\end{remark}


\begin{remark}
	One purpose of characterizing this class of divergences will be to identify new estimators which will be obtained as the minimizer of a divergence between an empirical estimate (see Remark~\ref{kernel}) of the true density $g$ and the model density $f_\theta$ in terms of the parameter $\theta$ over a suitable parameter space $\Theta$.  A natural follow up of the present work will be to look at properties of the minimum FDPD estimators from an overall standpoint and explore whether a general proof of asymptotic normality is possible under the presently existing conditions on the function $\varphi$, or under minimal additional conditions (apart from standard model conditions). 
	
\end{remark}

\begin{remark}{\label{kernel}}
	It may be noted that all minimum FDPD estimators are non-kernel divergence estimators in
	the sense of \cite{J01}, although not all minimum FDPD estimators are M-estimators. While the present
	paper is focused entirely on the characterization issue, eventually one would also like to know
	how useful are the inference procedures resulting from the minimization of divergences within
	the FDPD class (as already observed in the previous remark). In that respect the non-kernel divergence 
	property will lend a practical edge to the estimators and other inference procedures based on this family 
	in comparison with divergences which require an active use of a non-parametric smoothing technique in 
	their construction.
\end{remark}

Proposition~\ref{propn1} and Proposition~\ref{propn2} provide the complete characterization of the FDPD family and class of $\varphi$ functions generating them. We trust that this characterization describes the class within which one can search for suitable minimum divergence procedures exhibiting good balance between model efficiency and robustness.

\section{Acknowledgements}
We are grateful to two anonymous referees and the Associate Editor, whose suggestions have led to an improved version of the paper. In particular, it has allowed us to prove Proposition~\ref{propn2} without smoothness conditions (or even the assumption of continuity) on $\varphi$. Also, a comment about possible $\alpha$-specific $\varphi$ functions allowed us to make our result more general.

\section{Appendix}

The proof in Proposition~\ref{propn2} depends on some additional results involving $\lambda$-convex functions and general convex functions. These results, which are used as tools in our main pursuit, are presented here in the Appendix, separately, so as not to lose focus from our main characterization problem. 

\begin{definition}{\label{lambd}}
	Let $\lambda \in (0,1)$ and $-\infty \leq a <b \leq \infty.$ A function $f:(a,b) \to \mathbb{R}$ is said to be $\lambda$-convex if 
	$$ f(\lambda x + (1-\lambda) y ) \leq \lambda f(x) + (1-\lambda)f(y), \; \; \forall \; x,y \in (a,b).$$
\end{definition}

Obviously any convex function is also $\lambda$-convex, though the converse is not generally true. Traditionally, $1/2$-convex functions are  called midpoint convex. Under some further assumptions on $f$, like Lebesgue measurability or boundedness on a set with positive Lebesgue measure, one can prove that midpoint convex functions are indeed convex; see~\cite[Section I.3]{donoghue} for an extensive account of these kind of results. Here we shall prove a similar result in Proposition~\ref{lambdacon:main} for $\lambda$-convex functions for any $\lambda \in (0,1)$. Lemma~\ref{lem1} and Lemma~\ref{lem2} are instrumental in proving Proposition~\ref{lambdacon:main}.

\begin{lemma}{\label{lem1}}
	Suppose that $f:(a,b) \to \mathbb{R}$ is $\lambda$-convex. Then $f$ is continuous at $x \in (a,b)$ if and only if $f$ is bounded on an interval around $x$.
\end{lemma}

\begin{proof}
	The proof of the only if part is trivial from the definition of continuity. The proof of the if part is inspired by the proof of the theorem in~\cite[pp.12]{donoghue}. Suppose that $f$ is bounded on an interval around $x$. This condition can equivalently be written as 
	$$ - \infty < \liminf_{y \to x} f(y) \leq f(x) \leq  \limsup_{y \to x} f(y) < \infty.$$
	Applying $\lambda$-convexity for the function $f$ we can write the following.
	\begin{align*}
		\limsup_{y \to x} f(y) &= \limsup_{y \to x} f \left( \lambda \dfrac{y-(1-\lambda)x}{\lambda} + (1-\lambda)x\right) \\
		& \leq \limsup_{y \to x} \left[ \lambda f\left( \dfrac{y-(1-\lambda)x}{\lambda}\right) + (1-\lambda) f(x)\right] \\
		& = \lambda \limsup_{y \to x}  f\left( \dfrac{y-(1-\lambda)x}{\lambda}\right) + (1-\lambda) f(x) \leq \lambda \limsup_{y \to x} f(y) + (1-\lambda)f(x),
	\end{align*}
	where the last inequality follows from the observation that $(y-(1-\lambda)x)/\lambda$ converges to $x$ if $y$ converges to $x$. Since $\limsup_{y \to x} f(y)$ is finite, we can conclude that it is at most $f(x)$ and hence equal to $f(x)$. Applying $\lambda$-convexity again,
	\begin{align*}
		f(x) = \liminf_{y \to x} f \left( \lambda y + (1-\lambda) \dfrac{x-\lambda y}{1-\lambda} \right) &\leq \liminf_{y \to x} \left[ \lambda f(y) + (1-\lambda) f \left( \dfrac{x-\lambda y}{1-\lambda}\right) \right] \\
		& \leq \lambda \liminf_{y \to x} f(y) + (1-\lambda) \limsup_{y \to x}  f \left( \dfrac{x-\lambda y}{1-\lambda}\right) \\
		& \leq \lambda \liminf_{y \to x} f(y) + (1-\lambda) \limsup_{y \to x} f(y) \\
		& = \lambda \liminf_{y \to x} f(y) + (1-\lambda) f(x),
	\end{align*}
	implying that $\liminf_{y \to x} f(y)$ is at least $f(x)$ and hence equal to $f(x)$. In other words, both $\limsup_{y \to x} f(y)$ and $\liminf_{y \to x} f(y)$ are equal to $f(x)$ and therefore $f$ is continuous at $x$.
	
\end{proof}

\begin{lemma}{\label{lem2}}
	Suppose that $f:(a,b) \to \mathbb{R}$ is $\lambda$-convex and continuous. Then $f$ is convex.
\end{lemma}

\begin{proof}
	We shall prove the statement by contradiction. Suppose that $f$ is not convex. Then we can find $x,y \in (a,b)$ with $x <y$ and $\beta \in (0,1)$ such that 
	$$ f(\beta x + (1-\beta) y ) > \beta f(x) + (1-\beta) f(y).$$ 
	Define $h :[0,1] \to \mathbb{R}$ as follows.
	$$ h(t) := f(tx+(1-t)y) - tf(x) - (1-t)f(y), \; \; \forall \; t \in [0,1].$$
	Since $f$ is continuous, so is $h$ and $M:=\sup_{t \in [0,1]} h(t) \geq h(\beta) >0$. Let $\beta_0$ be the infimum of the set $\left\{t \in [0,1] : h(t)=M \right\}$, which is non-empty due to continuity of $h$. Continuity of $h$ also guarantees that $h(\beta_0)=M$; hence $\beta_0 \in (0,1)$ since $h(0)=h(1)=0$. Get $\delta >0$ such that $(\beta_0-\delta, \beta_0+\delta) \subseteq (0,1)$ and define
	$$ \beta_1:=\beta_0-\delta(1-\lambda), \;\beta_2 := \beta_0 + \delta \lambda, \;  u:= \beta_1 x +(1-\beta_1)y, \;\; v := \beta_2x + (1-\beta_2)y.$$ 
	Note that $0 < \beta_1 < \beta_0 < \beta_2 <1$ with $\lambda \beta_1 + (1-\lambda) \beta_2 = \beta_0$ and $\lambda u + (1-\lambda) v = \beta_0x + (1-\beta_0) y.$ We can, therefore, write the following series of inequalities.
	\begin{align*}
		M > \lambda h(\beta_1) + (1- \lambda)h(\beta_2) &=  \lambda \left[ f(u) - \beta_1f(x) - (1-\beta_1) f(y) \right] \\
		&\hspace{ 1in } + (1-\lambda)  \left[ f(v) - \beta_2f(x) - (1-\beta_2) f(y) \right] \\
		& = \lambda f(u) + (1-\lambda) f(v) - \beta_0f(x) - (1-\beta_0)f(y) \\
		& \geq f(\lambda u +(1-\lambda) v)  - \beta_0f(x) - (1-\beta_0)f(y) \\
		& = f(\beta_0x + (1-\beta_0) y)  - \beta_0f(x) - (1-\beta_0)f(y) = h(\beta_0)=M,
	\end{align*}
	where the left-most inequality follows from the fact that $\beta_1 < \beta_0$ and hence $h(\beta_1) <M$. This gives us a contradiction.
\end{proof}

The following proposition now follows readily from Lemma~\ref{lem1} and Lemma~\ref{lem2}.

\begin{proposition}{\label{lambdacon:main}}
	Suppose that $f:(a,b) \to \mathbb{R}$ is $\lambda$-convex. Moreover for any $x \in (a,b)$, $f$ is bounded on an interval around $x$. Then $f$ is convex.
\end{proposition}

\begin{lemma}{\label{lem3}}
	Let $f:(a,b) \to \mathbb{R}$ be non-decreasing. Suppose the left-hand limit function $f(\cdot -)$, defined as $f(x-) = \lim_{y \uparrow x} f(y)$ for all $x \in (a,b)$, is continuous. Then $f$ is also continuous; in particular $f(\cdot-)=f$. 
\end{lemma}

\begin{proof}
	It is enough to show that $f(\cdot-)=f$. Take any $x \in (a,b)$. Monotonicity of $f$ implies that $f(x-) \leq f(x)$. On the other hand, $f(y) \geq f(x)$ for all $ a < x <y < b$; and hence $f(y-) \geq f(x).$  Since $f(\cdot-)$ is continuous, we can take $y \downarrow x$ to conclude that $f(x-) \geq f(x).$ This shows $f(x-)=f(x).$
\end{proof}

\bibliographystyle{plain}
\bibliography{fdpd_newrevised_arxiv}

\end{document}